\documentclass{article}

\usepackage[utf8]{inputenc}
\usepackage[T1]{fontenc}
\usepackage{hyperref}
\usepackage{url}
\usepackage{booktabs}
\usepackage{amsfonts}
\usepackage{nicefrac}
\usepackage{microtype}
\usepackage{lipsum}
\usepackage{graphicx}
\usepackage{amsmath}
\usepackage{amssymb}
\usepackage{amsthm}
\usepackage{amscd}
\usepackage{alltt}
\usepackage{enumitem}
\usepackage[all]{xy}
\SelectTips{cm}{12}

\newtheorem{lemma}{Lemma}
\newtheorem{theorem}{Theorem}
\newtheorem*{theorem*}{Theorem}

\DeclareMathOperator\mat{M}
\newcommand{\op}{\mathrm{op}}
\DeclareMathOperator\elin{E}
\DeclareMathOperator\st{St}
\newcommand\stfree{\mathrm{St}^{\mathrm F}}
\newcommand\stabs{\overline{\mathrm{St}}}
\DeclareMathOperator\klin{K}
\DeclareMathOperator\stmap{st}
\DeclareMathOperator\glin{GL}
\DeclareMathOperator\diag{D}

\DeclareMathOperator\Coder{Coder}
\newcommand{\leqt}{\trianglelefteq}
\DeclareMathOperator\Ad{Ad}
\DeclareMathOperator\Aut{Aut}
\newcommand\eps{\varepsilon}
\newcommand\trp{\(^{\mathrm t}\)}

\newcommand{\up}[2]{{^{#1}\!{#2}}}

\newcommand{\fref}[1]{(\ref{#1})}

\title{Explicit presentation of relative Steinberg groups}

\author{
  Egor Voronetsky
  \thanks{Research is supported by the Russian Science Foundation grant 19-71-30002.} \\
  Chebyshev Laboratory, \\
  St. Petersburg State University, \\
  14th Line V.O., 29B, \\
  Saint Petersburg 199178 Russia \\
}

\begin{document}
\maketitle

\begin{abstract}
We find an explicit presentation of relative linear Steinberg groups \(\st(n, R, I)\) for any ring \(R\) and \(n \geq 4\) by generators and relations as abstract groups. We also prove a similar result for relative simply laced Steinberg groups \(\st(\Phi; R, I)\) for commutative \(R\) and \(\Phi \in \{\mathsf A_\ell, \mathsf D_\ell, \mathsf E_\ell \mid \ell \geq 3\}\).
\end{abstract}

\section{Introduction}

Relative linear Steinberg groups \(\st(n, R, I)\) are defined by F. Keune and J.-L. Loday in \cite{Keune, Loday} as certain crossed modules over the absolute Steinberg group \(\st(n, R)\), where \(R\) is an associative ring with \(1\) and \(I \leqt R\) is an ideal. They actually considered only the stable case \(n = \infty\), but the definition easily generalizes to the unstable situation. The relative Steinberg groups have explicit presentation by generators and relations, but as groups with an action of \(\st(n, R)\). There is another definition of \(\st(n, R, I)\) by M. Tulenbaev \cite{Tulenbaev} if \(R\) is commutative, but in terms of generators \(X_{vw}\) parametrized by vectors \(v, w\) satisfying some conditions instead of the elementary generators \(z_{ij}(a, p)\). By \cite{CentralityE}, both definitions give the same groups for commutative rings.

For relative simply laced Steinberg groups \(\st(\Phi; R, I)\) we use a definition from \cite{CentralityE}. It is well-known that relative Steinberg groups (both linear and simply laced) are generated by the elementary conjugates \(z_{ij}(a, p) = \up{x_{ji}(p)}{x_{ij}(a)}\) as abstract groups (see, for example, \cite[lemma 5]{CentralityE} for the case of simply-laced groups). By \cite[theorem 9]{CentralityE}, all relations in \(\st(\Phi; R, I)\) come from relative Steinberg groups associated with root subsystems of \(\Phi\) of type \(\mathsf A_3\), i.e. they involve a bounded number of roots. We give an explicit list of such relations (theorems \ref{rel-lin} and \ref{rel-ade}) and show, in particular, that they involve only roots of subsystems of types \(\mathsf A_1\), \(\mathsf A_1 \times \mathsf A_1\), and \(\mathsf A_2\). Our list consists of a finite number of identities parametrized by roots.

To prove our result in the linear case, it is useful to consider a slightly more general groups. It is straightforward to construct absolute and relative Steinberg groups associated with any associative ring \(R\) with a complete family of full orthogonal idempotents instead of a matrix algebra \(\mat(n, S)\). Below we use the notation \(\st(R)\) and \(\st(R, A)\) for such Steinberg groups, where \(A\) is an ideal or, more generally, a ring-theoretic crossed module over \(R\). They should not be confused with the stable Steinberg groups such as \(\varinjlim_n \st(n, R)\). In the commutative case we denote the base ring by \(K\) and its ideal (or a commutative crossed module) by \(\mathfrak a\).

As an application, we prove the following: the subgroup of \(\st(\Phi; K, \mathfrak a)\) generated by elementary generators \(x_\alpha(a)\) is normal and contains the commutator, and similarly for linear groups over arbitrary rings. See theorems \ref{f-trans-lin} and \ref{f-trans-chev} in the last section.

\section{Relative linear Steinberg groups}

We use the conventions \(\up gh = ghg^{-1}\) and \([g, h] = ghg^{-1}h^{-1}\) for group theoretical operations. If a group \(G\) acts on a group \(H\) by automorphisms, then we usually denote the action by \(\up gh\), i.e. as the conjugation in \(H \rtimes G\).

Let \(R\) be a unital associative ring. A non-unital \(R\)-algebra \(A\) is a non-unital associative ring and an \(R\)-\(R\)-bimodule with the same addition such that
\[p(ab) = (pa)b,\quad (ab)p = a(bp),\quad (ap)b = a(pb)\]
for all \(a, b \in A\) and \(p \in R\). As in group theory, we say that \(d \colon A \to R\) is a crossed module if \(A\) is a non-unital \(R\)-algebra, \(d\) is a homomorphism of non-unital \(R\)-algebras and \(ab = d(a)\, b = a\, d(b)\) for all \(a, b \in A\). For example, if \(I\) is an ideal of \(R\), then the inclusion map \(d \colon I \to R\) is a crossed module. Another examples are the ring homotopes used in \cite{LinK2}: if \(s \in R\) is a central element, then \(R^{(s)} = \{a^{(s)} \mid a \in R\}\) is a non-unital \(R\)-algebra with the operations
\[a^{(s)} + b^{(s)} = (a + b)^{(s)},\quad a^{(s)} b^{(s)} = (asb)^{(s)},\quad p a^{(s)} = (pa)^{(s)},\quad a^{(s)} p = (ap)^{(s)},\]
and the ring homomorphism \(d \colon R^{(s)} \to R, a^{(s)} \mapsto as\) is a crossed module. The centrality condition is essential.

If \(A\) is a non-unital \(R\)-algebra, then the semi-direct product \(A \rtimes R\) is a ring with an ideal \(A\). As an abelian group it is the direct sum of \(A\) and \(R\), the multiplication is given by
\[(a \oplus p)(b \oplus q) = (ab + aq + pb) \oplus pq.\]
In particular, \(A\) is a crossed module over \(A \rtimes R\).

From now on fix a unital ring \(R\) with a complete family of full orthogonal idempotents \(e_1, \ldots, e_n\) and a non-unital \(R\)-algebra \(A\). In other words, \(e_i^2 = e_i\), \(e_i e_j = 0\) for \(i \neq j\), \(\sum_{i = 1}^n e_i = 1\), and \(R = R e_i R\) as an ideal. An ``isotropic'' linear Steinberg group \(\st(R)\) is the abstract group with the generators \(x_{ij}(p)\) for \(1 \leq i, j \leq n\), \(i \neq j\), \(p \in e_i R e_j\) and the relations
\begin{align*}
x_{ij}(p)\, x_{ij}(q) &= x_{ij}(p + q); \tag{St1} \label{st1}\\
[x_{ij}(p), x_{jk}(q)] &= x_{ik}(pq) \text{ for } i \neq k; \tag{St2} \label{st2}\\
[x_{ij}(p), x_{kl}(q)] &= 1 \text{ for } j \neq k \text{ and } i \neq l; \tag{St3} \label{st3}
\end{align*}
There is a canonical group homomorphism \(\stmap \colon \st(R) \to R^*, x_{ij}(p) \mapsto 1 + p\). If \(R = \mat(n, S)\) is a matrix ring with the diagonal idempotents, then this Steinberg group may be identified with the usual unstable Steinberg group \(\st(n, S)\). An unrelativized Steinberg group \(\st(A)\) is given by the same presentation, but with \(a \in e_i A e_j\) for the generators \(x_{ij}(a)\).

Now we define \(\st'(R, A)\) as the group with an action of \(\st(R)\) generated by elements \(x_{ij}(a)\) for \(1 \leq i, j \leq n\), \(i \neq j\), \(a \in e_i A e_j\) with the relations \fref{st1}, \fref{st2}, \fref{st3} as in the unrelativized case and
\begin{align*}
\up{x_{ij}(p)}{x_{kl}(a)} &= x_{kl}(a) \text{ for } j \neq k \text{ and } l \neq i; \tag{Rel1} \label{rel1}\\
\up{x_{ij}(p)}{x_{jk}(a)} &= x_{ik}(pa)\, x_{jk}(a) \text{ for } i \neq k; \tag{Rel2} \label{rel2}\\
\up{x_{jk}(p)}{x_{ij}(a)} &= x_{ij}(a)\, x_{ik}(-ap) \text{ for } i \neq k. \tag{Rel2\trp} \label{rel2-t}
\end{align*}

Finally, if \(d \colon A \to R\) is a crossed module, then a relative Steinberg group \(\st(R, A)\) is the factor-group of \(\st'(R, A)\) by
\[
\up{x_{ij}(d(a))}g = x_{ij}(a)\, g\, x_{ij}(-a) \text{ for any } g \in \st'(R, A). \tag{Rel3} \label{rel3}
\]

There is a canonical isomorphism \(\st'(R, A) \cong \st(A \rtimes R, A)\) and a decomposition \(\st(A \rtimes R) \cong \st'(R, A) \rtimes \st(R)\) into a semi-direct product. If \(d \colon A \to R\) is a crossed module, then \(\st(R, A) \to \st(R), x_{ij}(a) \mapsto x_{ij}(d(a))\) is a crossed module in the sense of group theory and the sequence
\[\st(R, A) \to \st(R) \to \st(R / d(A)) \to 1\]
is exact.

The opposite ring \(R^\op\) also has a complete family of full orthogonal idempotents \(e_i^\op\) and \(A^\op\) is a non-unital \(R^\op\)-algebra. The transposition maps are the anti-isomorphisms
\begin{align*}
\st(R) \to \st(R^\op), x_{ij}(p) &\mapsto x_{ji}(p^\op);\\
\st'(R, A) \to \st'(R^\op, A^\op), x_{ij}(a) &\mapsto x_{ji}(a^\op).
\end{align*}
In other words, \(x_{ij}(p) \mapsto x_{ji}(-p^\op)\) and \(x_{ij}(a) \mapsto x_{ji}(-a^\op)\) are isomorphisms of these groups. If \(d \colon A \to R\) is a crossed module, then \(d^\op \colon A^\op \to R^\op\) is also a crossed module and there is a transposition map \(\st(R, A) \to \st(R^\op, A^\op)\).

Now we consider some identities in \(\st'(R, A)\) on the elements \(z_{ij}(a, p) = \up{x_{ji}(p)}{x_{ij}(a)}\) for \(a \in e_i A e_j\) and \(p \in e_j R e_i\). In order to write them compactly, we use the abbreviations
\begin{align*}
z_{i, j[k]}(a, b; p) &= z_{ij}(a, p)\, x_{ik}(b)\, x_{jk}(pb) \tag{Z2} \label{z2}\\
&= \up{x_{ji}(p)}{\bigl(x_{ij}(a)\, x_{ik}(b)\bigr)};\\
z_{[i]j, k}(a, b; p) &= z_{jk}(b, p)\, x_{ik}(a)\, x_{ij}(-ap) \tag{Z2\trp} \label{z2-t}\\
&= \up{x_{kj}(p)}{\bigl(x_{ik}(a)\, x_{jk}(b)\bigr)};\\
z_{i \oplus j, k}(a, b; p, q) &= z_{i, k[j]}(a, -aq; p)\, z_{j, k[i]}(b, -bp; q) \tag{Z4} \label{z4}\\
&= \up{x_{ki}(p)\, x_{kj}(q)}{\bigl(x_{ik}(a)\, x_{jk}(b)\bigr)};\\
z_{i, j \oplus k}(a, b; p, q) &= z_{[k]i, j}(qa, a; p)\, z_{[j]i, k}(pb, b; q) \tag{Z4\trp} \label{z4-t}\\
&= \up{x_{ji}(p)\, x_{ki}(q)}{\bigl(x_{ij}(a)\, x_{ik}(b)\bigr)}
\end{align*}
for distinct \(i, j, k\).

\begin{lemma}\label{z-rel}
The elements \(z_{ij}(a, p)\) satisfy the relations
\begin{align*}
z_{ij}(a + a', p) &= z_{ij}(a, p)\, z_{ij}(a', p); \tag{Add1} \label{add1}\\
z_{i, j[k]}(a + a', b + b'; p) &= z_{i, j[k]}(a, b; p)\, z_{i, j[k]}(a', b'; p); \tag{Add2} \label{add2}\\
z_{[i]j, k}(a + a', b + b'; p) &= z_{[i]j, k}(a, b; p)\, z_{[i]j, k}(a', b'; p); \tag{Add2\trp} \label{add2-t}\\
z_{i \oplus j, k}(a + a', b + b'; p, q) &= z_{i \oplus j, k}(a, b; p, q)\, z_{i \oplus j, k}(a', b'; p, q); \tag{Add3} \label{add3}\\
z_{i, j \oplus k}(a + a', b + b'; p, q) &= z_{i, j \oplus k}(a, b; p, q)\, z_{i, j \oplus k}(a', b'; p, q); \tag{Add3\trp} \label{add3-t}\\
\up{z_{ij}(c, p)}{\bigl(x_{ik}(a)\, x_{jk}(b)\bigr)} &= x_{ik}\bigl(a + cb - cpa\bigr)\, x_{jk}\bigl(b + pcb - pcpa\bigr); \tag{Conj1} \label{conj1}\\
\up{z_{ij}(c, p)}{\bigl(x_{ki}(a)\, x_{kj}(b)\bigr)} &= x_{ki}\bigl(a + acp + bpcp\bigr)\, x_{kj}\bigl(b - ac - bpc\bigr); \tag{Conj1\trp} \label{conj1-t}\\
\bigl[x_{jk}(pa)\, x_{ik}(a), x_{kj}(b)\, x_{ki}(-bp)\bigr] &= z_{ij}(ab, p) \text{ for } i \neq j; \tag{Mult} \label{mult}\\
\bigl[z_{ij}(a, p), z_{kl}(b, q)\bigr] &= 1 \text{ for } i, j, k, l \text{ distinct}; \tag{Dis} \label{dis}\\
z_{i \oplus j, k}(a, b; p, q) &= z_{j \oplus i, k}(b, a; q, p); \tag{Sym} \label{sym}\\
z_{i, j \oplus k}(a, b; p, q) &= z_{i, k \oplus j}(b, a; q, p); \tag{Sym\trp} \label{sym-t}\\
\up{z_{ij}(c, r)}{z_{i \oplus j, k}(a, b; p, q)} &= \up{x_{ki}(pcr + qrcr)\, x_{kj}(-pc - qrc)}{z_{i \oplus j, k}\bigl(a + cb - cra, b + rcb - rcra; p, q\bigr)}; \tag{Conj2} \label{conj2}\\
\up{z_{ij}(c, r)}{z_{k, i \oplus j}(a, b; p, q)} &= \up{x_{ik}(cq - crp)\, x_{jk}(rcq - rcrp)}{z_{k, i \oplus j}\bigl(a + acr + brcr, b - ac - brc; p, q\bigr)}; \tag{Conj2\trp} \label{conj2-t}\\
z_{i \oplus j, k}(a, qa; r - pq, p) &= z_{i, j \oplus k}(-ap, a; q, r). \tag{HW} \label{hw}
\end{align*}
in \(\st'(R, A)\). If \(d \colon A \to R\) is a crossed module, then they also satisfy a variant of \fref{rel3}
\[z_{ij}\bigl(a, p + d(b)\bigr) = x_{ji}(b)\, z_{ij}(a, p)\, x_{ji}(-b). \tag{Rel4} \label{rel4}\]
\end{lemma}
\begin{proof}
The relations \fref{add1}, \fref{add2}, \fref{add3}, \fref{conj1}, \fref{dis}, \fref{sym}, \fref{rel4} are obvious. The remaining ones may be written as
\begin{align*}
\bigl[\up{x_{ji}(p)}{x_{ik}(a)}, \up{x_{ji}(p)}{x_{kj}(b)}\bigr] &= \up{x_{ji}(p)}{x_{ij}(ab)}; \tag{Mult}\\
\up{z_{ij}(c, r)\, x_{ki}(p)\, x_{kj}(q)}{\bigl(x_{ik}(a)\, x_{jk}(b)\bigr)} &= \up{\up{z_{ij}(c, r)}{(x_{ki}(p)\, x_{kj}(q))}\, z_{ij}(c, r)}{\bigl(x_{ik}(a)\, x_{jk}(b)\bigr)}; \tag{Conj2}\\
\up{x_{ki}(r)\, \up{x_{ji}(q)}{x_{kj}(p)}\, x_{ji}(q)}{x_{ik}(a)} &= \up{x_{ki}(r)\, x_{ji}(q)\, x_{kj}(p)}{x_{ik}(a)}; \tag{HW}
\end{align*}
so they hold in \(\st'(R, A)\). Finally, the relations with the index \(t\) follow by transposition and the commutativity relations \fref{st3}, \fref{add2}, \fref{add2-t}, \fref{sym}, \fref{sym-t}. The relation \fref{hw} is equivalent to its own transpose modulo these commutativity relations.
\end{proof}

\section{Generators and relations}

Let \(\stfree(R, A)\) be the free group generated by the symbols \(z_{ij}(a, p)\), where \(1 \leq i, j \leq n\), \(i \neq j\), \(a \in e_i A e_j\), and \(p \in e_j R e_i\). A transposition map for this group is the anti-isomorphism
\[\stfree(R, A) \to \stfree(R^\op, A^\op), z_{ij}(a, p) \mapsto z_{ji}(a^\op, p^\op).\]
The relations from lemma \ref{z-rel} are formal identities in \(\stfree(R, A)\) if we use the notation \(x_{ij}(a) = z_{ij}(a, 0)\), \fref{z2}, \fref{z2-t}, \fref{z4}, and \fref{z4-t}. In the proofs that some identities from lemma \ref{z-rel} hold in a factor-group of \(\stfree(R, A)\) we usually consider only the non-transposed versions.

Now let \(\stabs'(R, A)\) be the factor-group of \(\stfree(R, A)\) by \fref{add1}, \fref{dis}, \fref{conj2}, \fref{conj2-t}, and \fref{hw}. If \(d \colon A \to R\) is a crossed module, then we define \(\stabs(R, A)\) as the factor-group of \(\stabs'(R, A)\) by \fref{rel4}.

\begin{lemma}\label{abs-rel}
The Steinberg relations and all the identities from lemma \ref{z-rel} except \fref{rel4} hold in \(\stabs'(R, A)\).
\end{lemma}
\begin{proof}
The identity \fref{conj1} follows from the case \(p = q = 0\) of \fref{conj2}, and \fref{mult} follows from the case \(a = q = r = 0\). The Steinberg relations easily follow from \fref{add1}, \fref{conj1}, and \fref{dis}. Also, \fref{add2} follows from \fref{conj1}, \fref{add1}, and the Steinberg relations.

Substituting \(b = ra\), \(q = 0\) in \fref{conj2} and applying \fref{st3}, we get \fref{sym-t}. Now \fref{add3} follows from \fref{add2} and \fref{sym}.
\end{proof}

By lemma \ref{z-rel} there is a canonical homomorphism \(\mu \colon \stabs'(R, A) \to \st'(R, A), z_{ij}(a, p) \mapsto z_{ij}(a, p)\). So there is a sequence
\[\stfree(R, A) \to \stabs'(R, A) \to \st'(R, A) \xrightarrow{\stmap} \glin(A)\]
and similarly for the relative groups if \(d \colon A \to R\) is a crossed module. Here \(\glin(A)\) is the group of quasi-invertible elements of \(A\), i.e. elements \(a \in A\) such that \(a + b + ab = a + b + ba = 0\) for some \(b \in A\), and \(\stmap \colon x_{ij}(a) \mapsto a\) is \(\st(R)\)-equivariant. The diagonal group
\[\diag(R) = \{r \in R^* \mid e_i r e_j = 0 \text{ for } i \neq j\}\]
acts on the this sequence and its relativized variant by
\[\up r{z_{ij}(a, p)} = z_{ij}(rar^{-1}, rpr^{-1}).\]

Our goal is to show that if \(n \geq 4\), then \(\st(R)\) acts on \(\overline \st'(R, A)\) and \(\mu\) is an isomorphism of groups with actions of \(\st(R)\). Before we start the proof it is useful to simplify calculations in \(\overline \st'(R, A)\).

Let
\[\Phi = \{\mathrm e_i - \mathrm e_j \mid 1 \leq i, j \leq n; i \neq j\} \subseteq \mathbb R^n\]
be the root system of type \(\mathsf A_{n - 1}\), so every symbol \(z_{ij}(-, =)\) in \(\stfree(R, A)\) is indexed by the root \(\mathrm e_j - \mathrm e_i\). A subset \(\Sigma \subseteq \Phi\) is called closed if \(\alpha, \beta \in \Sigma\), \(\alpha + \beta \in \Phi\) imply \(\alpha + \beta \in \Sigma\). A subset \(\Sigma \subseteq \Phi\) is called special closed, if it is closed and does not contain opposite roots. It is well-known that \(\Sigma\) is special closed if and only if it is closed and lies in an open half-space (with the boundary passing through \(0\)). We say that an element \(\alpha\) of a special closed \(\Sigma \subseteq \Phi\) is extreme if it is not a sum of two elements of \(\Sigma\). Hence if \(\alpha \in \Sigma\) is extreme, then \(\Sigma \setminus \{\alpha\} \subseteq \Phi\) is also a special closed subset.

For any \(\alpha = \mathrm e_i - \mathrm e_j \in \Phi\) we use the notation \(A_\alpha = e_i A e_j\), so \(x_\alpha(A_\alpha) = \{x_{ij}(a) \mid a \in A_\alpha\}\) is a subgroup of \(\st(A)\), \(\stabs'(R, A)\), or \(\stabs(R, A)\) depending on the context. If \(\Sigma \subseteq \Phi\) is a special closed subset, then the restriction of \(\st(A) \to \glin(A)\) to the subgroup \(\prod_{\alpha \in \Sigma} x_\alpha(A_\alpha)\) is injective, where the product is taken in any order. Moreover, the product map
\[\prod_{\alpha \in \Sigma} x_\alpha \colon \prod_{\alpha \in \Sigma} A_\alpha \to \prod_{\alpha \in \Sigma} x_\alpha(A_\alpha)\]
is one-to-one.

In the following result we use the notation \(\up g{\{h\}}\) in order to distinguish \(h \in \st(A)\) with its image in \(\stabs'(R, A)\).

\begin{lemma}\label{uni-conj}
Let \(\Sigma \subseteq \Phi\) be a special closed subset. Then for any \(g \in \prod_{\alpha \in \Sigma} x_\alpha(R_\alpha) \leq \st(R)\) there is a unique homomorphism \(\up g{\{-\}_\Sigma} \colon \st(A) \to \stabs'(R, A)\) such that
\begin{align*}
\up{g\, x_{kl}(p)}{\{x_{ij}(a)\}_\Sigma} &= \up g {\{x_{ij}(a)\}_\Sigma} \text{ for } l \neq i \text{ and } j \neq k;\\
\up{g\, x_{ki}(p)}{\{x_{ij}(a)\}_\Sigma} &= \up g {\{x_{kj}(pa)\, x_{ij}(a)\}_\Sigma} \text{ for } j \neq k;\\
\up{g\, x_{jl}(p)}{\{x_{ij}(a)\}_\Sigma} &= \up g {\{x_{ij}(a)\, x_{il}(-ap)\}_\Sigma} \text{ for } l \neq i;\\
\up 1{\{x_{ij}(a)\}_\Sigma} &= x_{ij}(a);\\
\up{x_{ji}(p)}{\{x_{ij}(a)\}_\Sigma} &= z_{ij}(a, p).
\end{align*}
If \(\Sigma' \subseteq \Phi\) is a special closed subset contained in \(\Sigma\) and \(g \in \prod_{\alpha \in \Sigma'} x_\alpha(R_\alpha)\), then \(\up g{\{-\}_\Sigma} = \up g{\{-\}_{\Sigma'}}\), so the index \(\Sigma\) in the notation is redundant.
\end{lemma}
\begin{proof}
We use induction on the size of \(\Sigma\), the case \(\Sigma = \varnothing\) is obvious. Suppose that for all proper \(\Sigma' \subseteq \Sigma\) such that \(\Sigma'\) is a special closed subset of \(\Phi\) we already have the homomorphisms from the statement. First of all, we define \(\up g{\{-\}_\Sigma}\) on the free group \(\stfree(A)\) generated by all symbols \(x_{ij}(a)\) for \(1 \leq i, j \leq n\), \(i \neq j\), and \(a \in e_i A e_j\).

Take any root \(\alpha = \mathrm e_i - \mathrm e_j \in \Phi\). Suppose that \(\Sigma\) contains an extreme root \(\beta_1 = \mathrm e_k - \mathrm e_l \neq -\alpha\), then any \(g \in \prod_{\beta \in \Sigma} x_\beta(R_\beta)\) has a decomposition \(g = g_1\, x_{kl}(p)\) for \(g_1 \in \prod_{\beta \in \Sigma'} x_\beta(R_\beta)\), where \(\Sigma' = \Sigma \setminus \beta_1\). Let
\begin{align*}
\up g{\{x_{ij}(a)\}_\Sigma} &= \up{g_1}{\{x_{ij}(a)\}_{\Sigma'}} \text{ for } l \neq i \text{ and } j \neq k;\\
\up g{\{x_{ij}(a)\}_\Sigma} &= \up{g_1}{\{x_{kj}(pa)\, x_{ij}(a)\}_{\Sigma'}} \text{ for } l = i \text{ and } j \neq k;\\
\up g{\{x_{ij}(a)\}_\Sigma} &= \up{g_1}{\{x_{ij}(a)\, x_{il}(-ap)\}_{\Sigma'}} \text{ for } l \neq i \text{ and } j = k.
\end{align*}

We have to check that this definition is independent on \(\beta_1\). Let \(\beta_2 \in \Sigma \setminus \{-\alpha, \beta_1\}\) be another extreme root. If \(\alpha \neq -\beta_1 - \beta_2\), then the definitions of \(\up g{\{x_{ij}(a)\}_\Sigma}\) via \(\beta_1\) and via \(\beta_2\) coincide. Indeed, the roots \(\alpha\), \(\beta_1\), \(\beta_2\) lie in a common special closed subset \(\Upsilon \subseteq \Phi\), so \(\up g{\{x_{ij}(a)\}_\Sigma} = \up{g_3}{\{h\}_{\Sigma''}}\) for some canonical \(g_3 \in \prod_{\beta \in \Sigma''} x_\beta(R_\beta)\) and \(h \in \prod_{\gamma \in \Upsilon} x_\gamma(A_\gamma)\), where \(\Sigma'' = \Sigma \setminus \{\beta_1, \beta_2\}\).

Hence we may assume that \(\alpha + \beta_1 + \beta_2 = 0\) and there are no other extreme roots of \(\Sigma\), so \(\Sigma = \{\beta_1, \beta_2, \beta_1 + \beta_2\}\). Without loss of generality, \(\alpha = \mathrm e_i - \mathrm e_k\), \(\beta_1 = \mathrm e_j - \mathrm e_i\), and \(\beta_2 = \mathrm e_k - \mathrm e_j\). In this case we use \fref{hw}:
\[\up{x_{ki}(r - pq)\, x_{kj}(p)}{\{x_{ik}(a)\, x_{jk}(qa)\}_{\Sigma \setminus \{\beta_1\}}} = \up{x_{ji}(q)\, x_{ki}(r)}{\{x_{ij}(-ap)\, x_{ik}(a)\}_{\Sigma \setminus \{\beta_2\}}}.\]
It remains to consider the case when there is no such \(\beta_1\). This means that \(\Sigma = \{-\alpha\}\), so we just define \(\up{x_{ji}(p)}{\{x_{ij}(a)\}_\Sigma} = z_{ij}(a, p)\).

By construction, \(\up g{\{h\}_\Sigma}\) is independent on \(\Sigma\). The uniqueness claim is also trivial.

Now let us check that \(\up g{\{-\}_\Sigma}\) factors through the Steinberg relations (i.e. maps them to identities in \(\stabs'(R, A)\)). For \fref{st1} this is obvious, so we only consider the Steinberg relation for \([x_{\alpha_1}(a), x_{\alpha_2}(b)]\), where \(\alpha_1\) and \(\alpha_2\) are linearly independent roots. If there is an extreme root \(\beta \in \Sigma\) such that \(\alpha_1, \alpha_2, \beta\) lie in a common special closed subset of \(\Phi\), then the claim follows from the construction of \(\up g{\{-\}_\Sigma}\) via \(\beta\). Otherwise let \(\Psi\) be the root subsystem generated by \(\alpha_1\), \(\alpha_2\), and \(\beta\), it is necessary of rank \(2\). If \(\Psi\) is of type \(\mathsf A_1 \times \mathsf A_1\), then we just apply \fref{dis}.

If \(\Psi\) is of type \(\mathsf A_2\) and \(\alpha_1 + \alpha_2 \notin \Phi\), then we apply \fref{sym} or \fref{sym-t}. So without loss of generality \(\alpha_1 = \mathrm e_i - \mathrm e_j\), \(\alpha_2 = \mathrm e_j - \mathrm e_k\), and \(\Sigma \subseteq \{-\alpha_1, -\alpha_2, -\alpha_1 - \alpha_2\}\). Now we do a calculation using \fref{conj2}:
\begin{align*}
\up{x_{ji}(p)\, x_{kj}(q)\, x_{ki}(r)}{\{x_{ij}(a)\, x_{jk}(b)\}_\Sigma}
&= \up{x_{ji}(p)}{\{x_{kj}(ra)\}_\Sigma}\, z_{ij}(a, p)\, \up{x_{ki}(r - qp)\, x_{kj}(q)}{\{x_{jk}(b)\}_\Sigma}\\
&= \up{x_{ki}(r - qp)\, x_{kj}(q)}{\{x_{ik}(ab)\, x_{jk}(b + pab)\}}\, \up{x_{ji}(p)}{\{x_{kj}(ra)\}}\, z_{ij}(a, p)\\
&= \up{x_{ji}(p)\, x_{kj}(q)\, x_{ki}(r)}{\{x_{ik}(ab)\, x_{jk}(b)\, x_{ij}(a)\}_\Sigma}.
\end{align*}

Finally, we have to prove the first three identities from the statement. Let \(\alpha \in \Phi\) and \(\beta \in \Sigma\) be roots appearing in the identity such that \(\alpha \neq -\beta\). Then there is an extreme \(\gamma \in \Sigma\) such that \(\alpha, \beta, \gamma\) lie in a common special closed subset of \(\Phi\). So we apply the construction of \(\up g{\{-\}_\Sigma}\) via \(\gamma\) and use the induction hypothesis.
\end{proof}

Hence in the group \(\stabs'(R, A)\) we have
\begin{align*}
z_{i, j[k]}(a, b; p) &= \up{x_{ji}(p)}{\{x_{ij}(a)\, x_{ik}(b)\}};\\
z_{[i]j, k}(a, b; p) &= \up{x_{kj}(p)}{\{x_{ik}(a)\, x_{jk}(b)\}};\\
z_{i \oplus j, k}(a, b; p, q) &= \up{x_{ki}(p)\, x_{kj}(q)}{\{x_{ik}(a)\, x_{jk}(b)\}};\\
z_{i, j \oplus k}(a, b; p, q) &= \up{x_{ji}(p)\, x_{ki}(q)}{\{x_{ij}(a)\, x_{ik}(b)\}}.
\end{align*}
Also, the diagonal group acts on the homomorphisms from lemma \ref{uni-conj} in the obvious way.

The next result allows us to identify \(\stabs'(R, A)\) with \(\stabs(A \rtimes R, A)\), so in all identities from lemma \ref{z-rel} we may assume that \(p, q, r \in A \rtimes R\) instead of \(p, q, r \in R\), and in lemma \ref{uni-conj} the element \(g\) may be taken in \(\prod_{\alpha \in \Sigma} x_\alpha(A_\alpha \rtimes R_\alpha) \leq \st(A \rtimes R)\).

\begin{lemma}\label{relative}
The natural homomorphism \(\stabs'(R, A) \to \stabs(A \rtimes R, A), z_{ij}(a, p) \mapsto z_{ij}(a, 0 \oplus p)\) is bijective.
\end{lemma}
\begin{proof}
Denote this homomorphism by \(\zeta\). The inverse homomorphism is given by
\[\xi \colon \stfree(A \rtimes R, A) \to \stabs'(R, A), z_{ij}(a, p_A \oplus p_R) \mapsto \up{x_{ji}(p_A)}{z_{ij}(a, p_R)}.\]
Clearly, \(\xi\) factors through \fref{add1}, \fref{dis}, and \fref{rel4}. Also,
\[\xi\bigl(z_{i \oplus j, k}(a, b; p_A \oplus p_R, q_A \oplus q_R)\bigr) = \up{x_{ki}(p_A)\, x_{kj}(q_A)}{z_{i \oplus j, k}(a, b; p_R, q_R)}.\]

For \fref{conj2} if the last two arguments of \(z_{i \oplus j, k}\) from the left hand side are in \(R\), then the image of the identity under \(\xi\) holds in \(\stabs'(R, A)\) by a triple application of \fref{conj2}. The general case reduces to this special case and \fref{conj1-t} by the formula for \(\xi\bigl(z_{i \oplus j, k}(a, b; p_A \oplus p_R, q_A \oplus q_R)\bigr)\) above.

Finally, for \fref{hw} we do the following calculation in \(\stabs'(R, A)\) using lemma \fref{uni-conj}:
\begin{align*}
\xi\bigl(z_{i \oplus j, k}\bigl(a, q_A a + q_R a;{} &(r_A \oplus r_R) - (p_A \oplus p_R) (q_A \oplus q_R), p_A \oplus p_R\bigr)\bigr)\\
&= \up{x_{ji}(q_A)\, x_{ki}(r_A - p_A q_R)\, x_{kj}(p_A)}{\bigl(\up{x_{ji}(q_R)\, x_{ki}(r_R)\, x_{kj}(p_R)}{\{x_{ik}(a)\}}\bigr)}\\
&= \xi\bigl(z_{i, j \oplus k}(-ap_A - ap_R, a; q_A \oplus q_R, r_A \oplus r_R)\bigr).
\end{align*}

Hence \(\xi \colon \stabs(A \rtimes R, A) \to \stabs'(R, A)\) is well-defined. Clearly, it is the inverse of \(\zeta\).
\end{proof}

Actually, the identities \fref{conj2} and \fref{conj2-t} in the definition of \(\stabs(R, A)\) may be replaced by
\begin{align*}
\up{z_{ij}(c, r)}{z_{i \oplus j, k}(a, b; p, q)} &= z_{i \oplus j, k}\bigl(a + cb - cra, b + rcb - rcra; p + d(pcr + qrcr), q - d(pc + qrc)\bigr); \tag{Conj2'} \label{conj2'}\\
\up{z_{ij}(c, r)}{z_{k, i \oplus j}(a, b; p, q)} &= z_{k, i \oplus j}\bigl(a + acr + brcr, b - ac - brc; p + d(cq - crp), q + d(rcq - rcrp)\bigr). \tag{Conj2\trp'} \label{conj2'-t}
\end{align*}
Indeed, they follow from the ordinary \fref{conj2} and \fref{conj2-t} by lemma \ref{relative}. Conversely, taking \(p = q = 0\) in \fref{conj2'} we obtain \fref{conj1}, and the Steinberg relations formally follow from \fref{add1}, \fref{conj1}, \fref{conj1-t}, and \fref{dis}. The identity \fref{conj2} follows from a combination of \fref{conj2'}, \fref{conj1}, the Steinberg relations, and \fref{rel4}.

\section{Root elimination: construction} \label{s-root-elim}

From now on we fix a crossed module \(d \colon A \to R\). All our further results have versions for the groups \(\stabs'\) associated with any non-unital \(R\)-algebra \(A\) by lemma \ref{relative}.

In order to construct an action of \(\st(R)\) on \(\stabs(R, A)\) we use the root elimination technique from \cite{LinK2}. Let \(\Psi \subseteq \Phi\) be a root subsystem (necessarily closed) with the span \(\mathbb R \Psi\). We denote the image of \(\Phi \setminus \Psi\) in \(\mathbb R^n / \mathbb R \Psi\) by \(\Phi / \Psi\) and the map \(\Phi \setminus \Psi \to \Phi / \Psi\) by \(\pi_\Psi\). Let \(e_i \sim e_j\) if \(i = j\) or \(\mathrm e_i - \mathrm e_j \in \Psi\), this is an equivalence relation on our family of idempotents in \(R\) since \(\Psi\) is closed and \(\Psi = -\Psi\). For an equivalence class \(\{e_i\}_{i \in I}\) let \(e_I = \sum_{i \in I} e_i\), these elements also form a complete family of full orthogonal idempotents. It follows that the set \(\Phi / \Psi\) itself is a root system of type \(\mathsf A_{n - 1 - m}\) with a suitable dot product, where \(m\) is the rank of \(\Psi\), it parametrizes the root subgroups for a new family of idempotents. Let us denote the groups associated with the new family by \(\st(A, \Phi / \Psi)\), \(\stfree(R, A; \Phi / \Psi)\), \(\stabs(R, A; \Phi / \Psi)\), \(\st(R, A; \Phi / \Psi)\), and \(\diag(R, \Phi / \Psi)\). For the groups associated with the old family of idempotents we add a parameter \(\Phi\) in the notation.

There is a one-to-one correspondence between root subsystems of \(\Phi\) containing \(\Psi\) and root subsystems of \(\Phi / \Psi\) given by the direct and inverse images under \(\pi_\Psi\). If \(\Xi \subseteq \Phi / \Psi\) is a root subsystem, then the root systems \((\Phi / \Psi) / \Xi\) and \(\Phi / \pi_\Psi^{-1}(\Xi)\) are canonically isomorphic. We write \(\Phi / \alpha\) instead of \(\Phi / \{-\alpha, \alpha\}\).

For a root \(\alpha = \mathrm e_l - \mathrm e_m\) consider the homomorphism
\begin{align*}
F_\alpha \colon \st(A, \Phi / \alpha) &\to \st(A, \Phi),\\
x_{ij}(a) &\mapsto x_{ij}(a) \text{ for } i, j \neq \infty;\\
x_{\infty j}(a) &\mapsto x_{lj}(e_l a)\, x_{mj}(e_m a);\\
x_{i \infty}(a) &\mapsto x_{il}(a e_l)\, x_{im}(a e_m);
\end{align*}
where \(e_\infty = e_l + e_m\) is the new idempotent. Clearly, it is well-defined. Composing these homomorphisms for various roots, we obtain homomorphisms \(F_\Psi \colon \st(A, \Phi / \Psi) \to \st(A, \Phi)\) for any root subsystem \(\Psi \subseteq \Phi\).

The following lemma is just a generalization of \fref{conj2} and \fref{conj2-t}.

\begin{lemma}\label{z-conj}
Let \(\Psi \subseteq \Phi\) be a root subsystem, \(\Sigma \subseteq \Phi / \Psi\) be a special closed subset. Then for all elements \(f \in \langle z_{\alpha}(A_\alpha, R_{-\alpha}) \mid \alpha \in \Psi \rangle \leq \stabs(R, A; \Phi)\), \(g \in \prod_{\beta \in \Sigma} x_\beta(R_\beta) \leq \st(R, \Phi / \Psi)\), \(h \in \st(A, \Phi / \Psi)\) we have
\[f\, \up{F_\Psi(g)}{\{F_\Psi(h)\}}\, f^{-1} = \up{F_\Psi(\up{\stmap(f)}{g})}{\bigl\{F_\Psi\bigl(\up{\stmap(f)}{h}\bigr)\bigr\}} \in \stabs(R, A; \Phi),\]
where \(\stmap(f)\) is the image of \(f\) in \(\glin(R)\).
\end{lemma}
\begin{proof}
Note that \(\stmap(f)\) actually lies in \(\diag(R, \Phi / \Psi)\), so it acts on the unrelativized Steinberg groups associated with \(\Phi / \Psi\). Also, \(\pi_\Psi^{-1}(\Sigma) \subseteq \Phi\) is a special closed subset, so the left hand side is defined. Without loss of generality, \(\Psi = \{-\alpha, \alpha\}\) for a root \(\alpha = \mathrm e_l - \mathrm e_m\) and \(f = z_{lm}(a, p)\). By the proof of lemma \ref{uni-conj} applied to \(\Phi / \alpha\), the element \(\up{F_\alpha(g)}{\{F_\alpha(h)\}}\) is a product of various \(\up{F_\alpha(x_{ji}(q))}{\{F_\alpha(x_{ij}(b))\}}\) for \(\mathrm e_i - \mathrm e_j \in \Phi / \alpha\) and such a decomposition preserves the action of \(\stmap(f)\) on the arguments, so it remains to consider the case \(\Sigma = \{\mathrm e_j - \mathrm e_i\}\), \(g = x_{ji}(q)\), and \(h = x_{ij}(b)\). Let \(e_\infty = e_l + e_m\) be the new idempotent. If \(i, j \neq \infty\), then the identity follows from \ref{dis}. Else this is precisely \ref{conj2} or \ref{conj2-t}.
\end{proof}

Now we define \(F_\Psi\) for the Steinberg groups \(\stabs\). For a root \(\alpha = \mathrm e_l - \mathrm e_m\) consider the homomorphism
\begin{align*}
F_\alpha \colon \stfree(R, A; \Phi / \alpha) &\to \stabs(R, A; \Phi);\\
z_{ij}(a, p) &\mapsto z_{ij}(a, p) \text{ for } i, j \neq \infty;\\
z_{\infty j}(a, p) &\mapsto z_{l \oplus m, j}(e_l a, e_m a; p e_l, p e_m);\\
z_{i \infty}(a, p) &\mapsto z_{i, l \oplus m}(a e_l, a e_m; e_l p, e_m p);
\end{align*}
where \(e_\infty = e_l + e_m\) is the new idempotent. The next lemma implies that \(F_\Psi \colon \stabs(R, A; \Phi / \Psi) \to \stabs(R, A; \Phi)\) is well-defined for any root subsystem \(\Psi \subseteq \Phi\).

\begin{lemma}\label{elim-map}
The homomorphism \(F_\alpha \colon \stabs(R, A; \Phi / \alpha) \to \stabs(R, A; \Phi)\) is well-defined for any root \(\alpha\). If \(\alpha, \beta \in \Phi\) are linearly independent and \(\Psi = \Phi \cap (\mathbb R\alpha + \mathbb R\beta)\), then
\[F_\alpha \circ F_{\pi_\alpha(\beta)} = F_\beta \circ F_{\pi_\beta(\alpha)} \colon \stabs(R, A; \Phi / \Psi) \to \stabs(R, A; \Phi).\]
\end{lemma}
\begin{proof}
Clearly, \(F_\alpha\) factors through \fref{add1}, \fref{dis}, and \fref{rel4}. Let \(\alpha = \mathrm e_l - \mathrm e_m\) and \(e_\infty = e_l + e_m\) be the new idempotent. Below all the indices are distinct.

We have
\begin{align*}
F_\alpha(z_{i \oplus j, k}(a, b; p, q)) &= \up{x_{ki}(p)\, x_{kj}(q)}{\{x_{ik}(a)\, x_{jk}(b)\}};\\
F_\alpha(z_{\infty \oplus j, k}(a, b; p, q)) &= \up{x_{kl}(p e_l)\, x_{km}(p e_m)\, x_{kj}(q)}{\{x_{lk}(e_l a)\, x_{mk}(e_m a)\, x_{jk}(b)\}};\\
F_\alpha(z_{i \oplus \infty, k}(a, b; p, q)) &= \up{x_{ki}(p)\, x_{kl}(q e_l)\, x_{km}(q e_m)}{\{x_{ik}(a)\, x_{lk}(e_l b)\, x_{mk}(e_m b)\}};\\
F_\alpha(z_{i \oplus j, \infty}(a, b; p, q)) &= \up{x_{li}(e_l p)\, x_{mi}(e_m p)\, x_{lj}(e_l q)\, x_{mj}(e_m q)}{\{x_{il}(a e_l)\, x_{im}(a e_m)\, x_{jl}(b e_l)\, x_{jm}(b e_m)\}};
\end{align*}
and the transposed versions directly by definitions.

Lemma \ref{z-conj} implies that \(F_\alpha\) factors through \fref{conj2}. For \fref{hw} we consider all nontrivial cases, i.e. the cases involving the index \(\infty\):
\begin{align*}
F_\alpha(z_{\infty \oplus j, k}(a, qa; r - pq, p))
&= \up{x_{jl}(q e_l)\, x_{jm}(q e_m)\, x_{kl}(r e_l)\, x_{km}(r e_m)\, x_{mj}(p)}{\{x_{lk}(e_l a)\, x_{mk}(e_m a)\}}\\
&= F_\alpha(z_{\infty, j \oplus k}(-ap, a; q, r));\\
F_\alpha(z_{i \oplus \infty, k}(a, qa; r - pq, p))
&= \up{x_{li}(e_l q)\, x_{mi}(e_m q)\, x_{ki}(r)\, x_{kl}(p e_l)\, x_{km}(p e_m)}{\{x_{ik}(a)\}}\\
&= F_\alpha(z_{i, \infty \oplus k}(-ap, a; q, r));\\
F_\alpha(z_{i \oplus j, \infty}(a, qa; r - pq, p))
&= \up{x_{ji}(q)\, x_{li}(e_l r)\, x_{mi}(e_m r)\, x_{lj}(e_l p)\, x_{mj}(e_m p)}{\{x_{il}(a e_l)\, x_{im}(a e_m)\}}\\
&= F_\alpha(z_{i, j \oplus \infty}(-ap, a; q, r)).
\end{align*}

The last claim follows from the formulas for \(F_\alpha(z_{i \oplus j, k}(a, b; p, q))\), where the indices may coincide with \(\infty\).
\end{proof}

There is also a useful formula
\[F_\Psi\bigl(\up g{\{h\}}\bigr) = \up{F_\Psi(g)}{\{F_\Psi(h)\}} \in \stabs(R, A),\]
where \(\Psi \subseteq \Phi\) is a root subsystem, \(g \in \prod_{\alpha \in \Sigma} x_\alpha(R_\alpha) \leq \st(R, \Phi / \Psi)\) for some special closed \(\Sigma \subseteq \Phi / \Psi\), and \(h \in \st(A, \Phi / \Psi)\).

\begin{lemma}\label{elim-sur}
Let \(\alpha \in \Phi\) be a root and suppose that \(n \geq 3\). Then \(F_\alpha \colon \stabs(R, A; \Phi / \alpha) \to \stabs(R, A; \Phi)\) is surjective.
\end{lemma}
\begin{proof}
Let \(\alpha = \mathrm e_l - \mathrm e_m\) and \(e_\infty = e_l + e_m\). We have to show that every generator \(z_{ij}(a, p)\) of \(\stabs(R, A; \Phi)\) lies in the image of \(F_\alpha\). If \(i, j, l, m\) are distinct, then \(z_{ij}(a, p) = F_\alpha\bigl(z_{ij}(a, p)\bigr)\). If \(i \in \{l, m\}\) and \(j \notin \{l, m\}\), then \(z_{ij}(a, p) = F_\alpha\bigl(z_{\infty j}(a, p)\bigr)\). The case \(i \notin \{l, m\}\) and \(j \in \{l, m\}\) follows by transposition.

If \(\{i, j\} = \{l, m\}\), then choose an index \(k\) different from \(i\) and \(j\). We have
\[z_{ij}(-ap, q) = F_\alpha\bigl(z_{\infty k}(a + qa, p - pq)\, x_{\infty k}(-a - qa)\bigr)\]
by \fref{hw} for \(a \in e_i A e_k\) and \(p \in e_k R e_j\). It follows that \(z_{ij}(a, p)\) lies in the image if \(a \in e_i A e_k R e_j\). But the idempotent \(e_k\) is full, hence \(e_i A e_k R e_j = e_i A e_j\).
\end{proof}

\section{Root elimination: bijectivity}

\begin{lemma} \label{elim-bij}
Let \(\alpha = \mathrm e_l - \mathrm e_m \in \Phi\) be a root and suppose that \(n \geq 4\). Then \(F_\alpha \colon \stabs(R, A; \Phi / \alpha) \to \stabs(R, A; \Phi)\) is an isomorphism. Also, \(F_\alpha\bigl(\up{\stmap(x_{ml}(p))}{F_\alpha^{-1}(x_{lm}(a))}\bigr) = z_{lm}(a, p)\) in \(\stabs(R, A; \Phi)\).
\end{lemma}
\begin{proof}
Let \(e_\infty = e_l + e_m\). We have to find the preimages \(\widetilde z_{ij}(a, p) \in \stabs(R, A; \Phi / \alpha)\) of \(z_{ij}(a, p)\) and to prove the defining identities for them. Note that the elements \(x_{lm}(p)\) and \(x_{ml}(q)\) for \(p, q \in R\) act on \(\stabs(R, A; \Phi / \alpha)\) via their images in \(\diag(R; \Phi / \alpha)\). Let
\begin{align*}
\widetilde z_{ij}(a, p) &= z_{ij}(a, p) \text{ if } i, j, l, m \text{ are distinct}, a \in e_i A e_j, p \in e_j R e_i;\\
\widetilde z_{ij}(a, p) &= z_{\infty j}(a, p) \text{ if } i \in \{l, m\}, j \notin \{l, m\}, a \in e_i A e_j, p \in e_j R e_i;\\
\widetilde z_{ij}(a, p) &= z_{i \infty}(a, p) \text{ if } i \notin \{l, m\}, j \in \{l, m\}, a \in e_i A e_j, p \in e_j R e_i;\\
\widetilde x_{ij}^k(a, p) &= z_{\infty k}(-a, p)\, x_{\infty k}(a) \text{ if } \{i, j\} = \{l, m\}, k \notin \{l, m\}, a \in e_i A e_k, p \in e_k R e_j;\\
\up k{\widetilde x_{ij}}(q, a) &= z_{k \infty}(a, q)\, x_{k \infty}(-a) \text{ if } \{i, j\} = \{l, m\}, k \notin \{l, m\}, q \in e_i R e_k, a \in e_k A e_j.
\end{align*}
The idea is to find elements \(\widetilde x_{ij}(a)\) for \(\{i, j\} = \{l, m\}\) such that \(\widetilde x_{ij}^k(a, p) = \widetilde x_{ij}(ap)\) and \(\up k{\widetilde x_{ij}}(q, a) = \widetilde x_{ij}(qa)\), and then define \(\widetilde z_{ij}(a, r) = \up{\stmap(x_{ji}(r))}{\widetilde x_{ij}(a)}\). All calculations below are in \(\stabs(R, A; \Phi / \alpha)\).

Let \(\{i, j\} = \{l, m\}\). It is easy to see that \(\widetilde x^k_{ij}(a, p)\) acts on \(\stabs(R, A; \Phi / \alpha)\) by conjugation in the same way as its image in \(\diag(R, \Phi / \alpha)\). Indeed, these actions coincide on the image of \(F_\beta\) for \(\beta = \mathrm e_\infty - \mathrm e_k\) by lemma \ref{z-conj}, and \(F_\beta\) is surjective by lemma \ref{elim-sur}. In particular, \(\bigl[\widetilde x^k_{ij}(a, p), \widetilde x^o_{ij}(b, q)\bigr] = 1\) for all \(k, o \notin \{l, m\}\). Also,
\[\widetilde x^k_{ij}(a, p)\, \widetilde x^k_{ij}(b, p) = z_{\infty k}(-a, p)\, z_{\infty k}(-b, p)\, x_{\infty k}(b)\, x_{\infty k}(a) = \widetilde x^k_{ij}(a + b, p).\]
By the same argument, the factors in the definitions of \(\widetilde x^k_{ij}\) and \(\up k{\widetilde x_{ij}}\) commute.

If \(o \notin \{k, l, m\}\), \(a \in e_i R e_k\), \(p \in e_k R e_o\), \(q \in e_o R e_j\), \(r \in e_k R e_j\), then \(qa = ra = 0\). We apply \fref{hw} to get
\begin{align*}
\widetilde x^k_{ij}(a, r + pq)
&= z_{\infty k}(-a, r + pq)\, x_{\infty k}(a)\\
&= z_{\infty o}(-ap, q)\, z_{\infty k}(-a, r)\, x_{\infty o}(ap)\, x_{\infty k}(a)\\
&= z_{\infty o}(-ap, q)\, x_{\infty o}(ap)\, z_{\infty k}(-a, r)\, x_{\infty k}(a) = \widetilde x^o_{ij}(ap, q)\, \widetilde x^k_{ij}(a, r).
\end{align*}
In other words, \(\widetilde x^k_{ij}(a, pq) = \widetilde x^o_{ij}(ap, q)\) and \(\widetilde x^k_{ij}(a, r + pq) = \widetilde x^k_{ij}(a, r)\, \widetilde x^k_{ij}(a, pq)\) if \(p \in e_k R e_o\), \(q \in e_o R e_j\), and \(k \neq o\). Since the idempotent \(e_o\) is full, \(\widetilde x^k_{ij}(a, r + r') = \widetilde x^k_{ij}(a, r)\, \widetilde x^k_{ij}(a, r')\) for all \(r, r' \in e_k R e_j\). It follows that \(\widetilde x^k_{ij}(a, pq) = \widetilde x^k_{ij}(ap, q)\) for \(p \in e_k R e_k\) and \(q \in e_k R e_j\).

Since \(e_k\) is a full idempotent, the homomorphism \(e_i A e_k \otimes_{e_k R e_k} e_k R e_j \to e_i A e_j, a \otimes p \mapsto ap\) is bijective. Hence there is a unique homomorphism \(\widetilde x_{ij} \colon e_i A e_j \to \stabs(R, A; \Phi / \alpha)\) such that \(\widetilde x^k_{ij}(a, p) = \widetilde x_{ij}(ap)\). Clearly, \(\widetilde x^o_{ij}(a, p) = \widetilde x_{ij}(a, p)\) for all indices \(o \notin \{l, m\}\).

Let \(k, o\) be distinct indices not in \(\{i, j\} = \{l, m\}\), \(q \in e_i R e_k\), \(a \in e_k A e_o\), \(p \in e_o R e_j\), so \(pq = 0\). Applying \fref{hw} with \(r = 0\) we get
\[
\up k{\widetilde x_{ij}}(q, ap) = x_{k \infty}(-ap)\, \up{x_{ko}(a)}{z_{k \infty}(ap, q)} = z_{\infty o}(-qa, p)\, x_{\infty o}(qa) = \widetilde x^o_{ij}(qa, p).
\]
Since \(\up k{\widetilde x_{ij}}(q, a)\) is also additive on \(q\) and \(a\), we get \(\up k{\widetilde x_{ij}}(q, a) = \widetilde x_{ij}(qa)\). Let \(\widetilde z_{ij}(a, r) = \up{\stmap(x_{ji}(r))}{\widetilde x_{ij}(a)}\), where \(\stmap(x_{ji}(r))\) is the image of \(x_{ji}(r)\) in \(\diag(R, \Phi / \alpha)\).

So there is a homomorphism \(G_\alpha \colon \stfree(R, A; \Phi) \to \stabs(R, A; \Phi / \alpha), z_{ij}(a, p) \mapsto \widetilde z_{ij}(a, p)\). If \(\{i, j\} = \{l, m\}\), then the element \(G_\alpha\bigl(\widetilde z_{ij}(a, p)\bigr)\) acts on \(\stabs(R, A; \Phi / \alpha)\) in the same way as \((1 + p)(1 + d(a))(1 - p) \in \diag(R, \Phi / \alpha)\). We have to check that \(G_\alpha\) factors through the defining identities of \(\stabs(R, A; \Phi)\). This is clear for \fref{add1}, \fref{dis}, \fref{rel4}, and the Steinberg relations. Since the definition of \(G_\alpha\) is invariant under transposition, it suffices to consider only the non-transposed identities.

The cases of the remaining identities where at most one index is in \(\{l, m\}\) are obvious, in the remaining ones two of the indices \(i, j, k\) coincide with \(l, m\). Let \(\diag_{ij}\) be the subgroup of \(\stfree(R, A; \Phi)\) generated by all \(z_{ij}(a, p)\) and \(z_{ji}(b, q)\), \(\Psi = \{\pm(\mathrm e_i - \mathrm e_j), \pm(\mathrm e_i - \mathrm e_k), \pm(\mathrm e_j - \mathrm e_k)\} \subseteq \Phi\) be a root subsystem, \(e_{\infty'} = e_i + e_j + e_k\) be the corresponding idempotent. Choose an index \(o \notin \{i, j, k\}\).

If \(\{j, k\} = \{l, m\}\), then
\[G_\alpha(z_{i \oplus j, k}(a, b; p, q)) = \up{1 + q}{\bigl(\up{x_{\infty i}(p)}{\{x_{i \infty}(a)\}}\, \widetilde x_{jk}(b)\, x_{\infty i}(-bp)\bigr)} = \up{1 + q}{G_\alpha(z_{[j]i, k}(b, a; p))}. \tag{G1} \label{g1}\]
Hence under this assumption and for \(g \in \diag_{ij}\), \(u = \stmap(g) \in \glin(R)\), \(r \in e_i R e_o\), \(a \in e_o R e_k\), \(p \in e_k R e_i\) we have
\begin{align*}
G_\alpha\bigl(\up g{z_{ik}(ra, p)}\bigr)
&= F_{\Psi / \alpha}\bigl(\up{u (1 + p)}{\bigl(\up{x_{\infty' o}(r)}{\{x_{o \infty'}(a)\}}\, x_{o \infty'}(-a)\bigr)}\bigr)\\
&= \up{1 + pu^{-1} e_j}{F_{\Psi / \alpha}\bigl(\up{x_{\infty' o}(ur + pu^{-1} e_i ur)}{\{x_{o \infty'}(a - apu^{-1} e_i)\}}\, x_{o \infty'}(apu^{-1} e_i - a)\bigr)}\\
&= \up{1 + pu^{-1} e_j}{\bigl(\up{x_{\infty i}(pu^{-1} e_i)\, x_{\infty o}(e_j ur)\, x_{io}(e_i ur)}{\{x_{o \infty}(a)\}}\, \up{x_{\infty i}(pu^{-1} e_i)}{\{x_{o \infty}(-a)\}}\bigr)}\\
&= \up{1 + pu^{-1} e_j}{\bigl(z_{i \infty}(e_i ura, pu^{-1} e_i)\, \widetilde x_{jk}(e_j ura)\, x_{\infty i}(-e_j urapu^{-1} e_i)\bigr)}\\
&= G_\alpha\bigl(z_{i \oplus j, k}(e_i ura, e_j ura; pu^{-1} e_i, pu^{-1} e_j)\bigr). \label{g2} \tag{G2}
\end{align*}

Similarly, if \(g \in \diag_{ij}\), \(u = \stmap(g)\), \(r \in e_j R e_o\), \(b \in e_o A e_k\), \(q \in e_k R e_j\), then
\begin{align*}
G_\alpha\bigl(\up g{z_{jk}(rb, q)}\bigr)
&= F_{\Psi / \alpha}\bigl(\up{u(1 + q)}{\bigl(\up{x_{\infty' o}(r)}{\{x_{o \infty'}(b)\}}\, x_{o \infty'}(-b)\bigr)}\bigr)\\
&= \up{1 + qu^{-1} e_j}{F_{\Psi / \alpha}\bigl(\up{x_{\infty' o}(ur + qu^{-1} e_i ur)}{\{x_{o \infty'}(b - bqu^{-1} e_i)\}}\, x_{o \infty'}(bqu^{-1} e_i - b)\bigr)}\\
&= \up{1 + qu^{-1} e_j}{\bigl(\up{x_{\infty i}(qu^{-1} e_i)\, x_{\infty o}(e_j ur)\, x_{i o}(e_i ur)}{\{x_{o \infty}(b)\}}\, \up{x_{\infty i}(qu^{-1} e_i)}{\{x_{o \infty}(-b)\}}\bigr)}\\
&= \up{1 + qu^{-1} e_j}{\bigl(z_{i \infty}(e_i urb, qu^{-1} e_i)\, \widetilde x_{jk}(e_j urb)\, x_{\infty i}(-e_j urbqu^{-1} e_i)\bigr)}\\
&= G_\alpha\bigl(z_{i \oplus j, k}(e_i urb, e_j urb; qu^{-1} e_i, qu^{-1} e_j)\bigr). \label{g3} \tag{G3}
\end{align*}
It follows that \(G_\alpha\) factors through \fref{conj1} and, consequently, through \fref{add2}.

Now we prove that \(G_\alpha\) factors through \fref{sym} in the form
\[\bigl[z_{ik}(a, p)\, x_{ij}(-aq)\, x_{kj}(-paq), z_{jk}(b, q)\, x_{ji}(-bp)\, x_{ki}(-qbp)\bigr] = 1.\]
If \(\{j, k\} = \{l, m\}\), then for \(r \in e_j A e_o\), \(b \in e_o A e_k\), \(q = 0\) we have
\begin{align*}
G_\alpha\bigl(\up{z_{ik}(a, p)}{\bigl(x_{jk}(rb)\, x_{ji}(-rbp)\bigr)}\bigr)
&= \up{\up{x_{\infty i}(p)}{\{x_{i \infty}(a)\}}}{\bigl(\up{x_{\infty o}(r)\, x_{\infty i}(p)}{\{x_{o \infty}(b)\}}\, \up{x_{\infty i}(p)}{\{x_{o \infty}(-b)\}}\bigr)}\\
&= \up{x_{\infty o}(r)\, x_{\infty i}(p)}{\{x_{o \infty}(b)\}}\, \up{x_{\infty i}(p)}{\{x_{o \infty}(-b)\}}\\
&= G_\alpha\bigl(x_{jk}(rb)\, x_{ji}(-rbp)\bigr),
\end{align*}
and the general case follows by using that \(G_\alpha\) factors through \fref{add2} and applying a conjugation by \(1 + q \in \glin(R)\). The case \(\{i, k\} = \{l, m\}\) follows by symmetry. If \(\{i, j\} = \{l, m\}\), then for \(r \in e_j R e_o\) and \(b \in e_o R e_k\) we have
\begin{align*}
G_\alpha\bigl(z_{jk}(rb, q)\, x_{ji}(-rbp)\, x_{ki}(-qrbp)\bigr)
&= \up{x_{k \infty}(q)\, x_{\infty o}(r)}{\{x_{ok}(b)\}}\, \widetilde x_{ji}(-rbp)\, x_{k \infty}(-qrbp)\, x_{ok}(-b)\, x_{o \infty}(bq)\\
&= \up{x_{k \infty}(p + q)\, x_{\infty o}(r)}{\{x_{ok}(b)\}}\, \up{x_{k \infty}(p + q)}{\{x_{ok}(-b)\}}\\
&= z_{\infty k}(rb, p + q),
\end{align*}
so
\[G_\alpha(z_{i \oplus j, k}(a, b; p, q)) = z_{\infty k}(a + b, p + q) \label{g4} \tag{G4}\]
and \(G_\alpha\) factors through the remaining case of \fref{sym}. Consequently, it factors through \fref{add3}.

Now it is easy to see that \(G_\alpha\) factors through \fref{conj2'} in the form
\[G_\alpha\bigl(\up g{z_{i \oplus j, k}(a, b; p, q)}\bigr) = G_\alpha\bigl(z_{i \oplus j, k}(e_i u(a + b), e_j u(a + b); (p + q) u^{-1} e_i, (p + q) u^{-1} e_j)\bigr)\]
for \(g \in \diag_{ij}\) and \(u = \stmap(g)\). If \(\{i, j\} = \{l, m\}\), this follows from \fref{g4}, and the other two cases follows from \fref{g2} and \fref{g3} by using symmetry and that \(G_\alpha\) factors through \fref{add3}.

It remains to prove that \(G_\alpha\) factors through \fref{hw}. The cases with \(p = 0\) or \(q = 0\) formally follow from the other relations, and the cases \(\{i, j\} = \{l, m\}\), \(\{j, k\} = \{l, m\}\) follow from them by conjugation using \fref{g1} and \fref{g4}. Finally, if \(\{i, k\} = \{l, m\}\), \(s \in e_i R e_o\), \(a \in e_o R e_k\), then
\begin{align*}
G_\alpha\bigl(z_{i \oplus j, k}(sa, qsa; r - pq, p)\bigr)
&= \up{1 + r - pq}{\bigl(z_{j \infty}(qsa, p)\, \widetilde x_{ik}(sa)\, x_{\infty j}(-sap)\bigr)}\\
&= \up{1 + r - pq}{\bigl(\up{x_{\infty j}(p)\, x_{\infty o}(s)\, x_{jo}(qs)}{\{x_{o \infty}(a)\}}\, \up{x_{\infty j}(p)}{\{x_{o \infty}(-a)\}}\bigr)}\\
&= F_{\Psi / \alpha}\bigl(\up{1 + p + q + r}{\bigl(\up{x_{\infty' o}(s)}{\{x_{o \infty'}(a)\}}\, x_{o \infty'}(-a)\bigr)}\bigr)\\
&= \up{1 + r}{\bigl(\up{x_{j \infty}(q)\, x_{\infty j}(p)\, x_{\infty o}(s)}{\{x_{o \infty}(a)\}}\, \up{x_{j \infty}(q)\, x_{\infty j}(p)}{\{x_{o \infty}(-a)\}}\bigr)}\\
&= \up{1 + r}{\bigl(z_{\infty j}(-sap, q)\, \widetilde x_{ik}(sa)\, x_{j \infty}(qsa)\bigr)}\\
&= G_\alpha\bigl(z_{i, j \oplus k}(-sap, sa; q, r)\bigr).
\end{align*}

Hence \(G_\alpha \colon \stabs(R, A; \Phi) \to \stabs(R, A; \Phi / \alpha)\) is well-defined. But \(F_\alpha\) is surjective by lemma \ref{elim-sur} and obviously \(G_\alpha \circ F_\alpha\) is the identity, so \(F_\alpha\) is an isomorphism with the inverse \(G_\alpha\). The second claim follows from \(G_\alpha(z_{lm}(a, p)) = \up{1 + p}{G_\alpha(x_{lm}(a))}\).
\end{proof}

Now we are ready to prove our first main result.

\begin{theorem}\label{rel-lin}
Let \(R\) be a ring with a complete family of \(n\) full orthogonal idempotents, \(A\) be a non-unital \(R\)-algebra. Then the homomorphism \(\mu \colon \stabs'(R, A) \to \st'(R, A)\) is surjective for \(n \geq 3\) and bijective for \(n \geq 4\). If \(d \colon A \to R\) is a crossed module, then the induced homomorphism \(\stabs(R, A) \to \st(R, A)\) is also surjective or bijective.
\end{theorem}
\begin{proof}
By lemma \ref{relative}, it suffices to consider the case of a crossed module. Suppose that \(n \geq 3\) and take \(g \in \stabs(R, A)\). Surjectivity means that \(\up{x_{ij}(p)}{\mu(g)} \in \st(R, A)\) lies in the image of \(\mu\) for all generators \(x_{ij}(p) \in \st(R)\). But this follows from lemma \ref{elim-sur} applied to the root \(\mathrm e_i - \mathrm e_j\).

Now suppose that \(n \geq 4\). We construct an action of \(\st(R)\) on \(\stabs(R, A) = \stabs(R, A; \Phi)\) as follows. For a root \(\alpha = \mathrm e_i - \mathrm e_j\) and \(p \in e_i R e_j\) let
\[\up{x_{ij}(p)}{F_\alpha(g)} = F_\alpha\bigl(\up{1 + p}g\bigr),\]
this gives an automorphism of \(\stabs(R, A)\) by lemma \ref{elim-bij}. Hence we have a homomorphism \(\Ad \colon \stfree(R) \to \Aut(\stabs(R, A))\), where \(\stfree(R)\) is the free group with generators \(x_{ij}(p)\).

If \(\Psi \subseteq \Phi\) is a root subsystem of rank at most \(2\), then any element from \(\langle x_\alpha(R_\alpha) \mid \alpha \in \Psi \rangle \leq \st(R)\) acts on the image of \(F_\Psi\) independently on its decomposition into a product of generators with roots in \(\Psi\), also \(F_\Psi \colon \stabs(R, A; \Phi / \Psi) \to \stabs(R, A; \Phi)\) is surjective by lemma \ref{elim-sur}. Hence \(\Ad\) is well-defined on \(\st(R)\).

Obviously, \(\mu \colon \stabs(R, A) \to \st(R, A)\) is \(\st(R)\)-equivariant. Also, there is a homomorphism
\[\nu \colon \st(R, A) \to \stabs(R, A), \up g{x_{ij}(a)} \mapsto \up g{x_{ij}(a)}.\]
By the second claim of lemma \ref{elim-bij}, \(\nu \circ \mu\) is the identity. Hence \(\mu\) is an isomorphism.
\end{proof}

\section{Relative simply laced Steinberg groups}

In this section \(K\) is a commutative unital ring and \(\mathfrak a\) is a non-unital commutative \(K\)-algebra (satisfying the identity \(ap = pa\) for \(a \in \mathfrak a\) and \(p \in K\)). Let \(\Phi\) be a crystallographic reduced irreducible simply laced root system, i.e. either \(\mathsf A_\ell\) for \(\ell \geq 1\), or \(\mathsf D_\ell\) for \(\ell \geq 4\), or \(\mathsf E_\ell\) for \(\ell \in \{6, 7, 8\}\). A corresponding Steinberg group \(\st(\Phi, K)\) is the abstract group with generators \(x_\alpha(p)\) for \(\alpha \in \Phi\), \(p \in K\) and the relations
\begin{align*}
x_\alpha(p)\, x_\alpha(q) &= x_\alpha(p + q); \tag{St1}\\
[x_\alpha(p), x_\beta(q)] &= x_{\alpha + \beta}(N_{\alpha \beta}\, pq) \text{ for } \alpha + \beta \in \Phi; \tag{St2}\\
[x_\alpha(p), x_\beta(q)] &= 1 \text{ for } \alpha + \beta \notin \Phi \cup \{0\}. \tag{St3}
\end{align*}
Here \(N_{\alpha \beta} \in \{-1, 1\}\) are the so-called structure constants, they are determined by the corresponding Chevalley group scheme over \(\mathrm{Spec}(\mathbb Z)\) up to a choice of parametrizations of the root subgroups. The unrelativized Steinberg group \(\st(\Phi, \mathfrak a)\) is defined in the same way, but with the parameters in \(\mathfrak a\).

We define a group \(\st'(\Phi; K, \mathfrak a)\) as the group with an action of \(\st(\Phi, K)\) generated by the elements \(x_\alpha(a)\) for \(\alpha \in \Phi\), \(a \in \mathfrak a\) satisfying the Steinberg relations and
\begin{align*}
\up{x_\alpha(p)}{x_\beta(a)} &= x_\beta(a) \text{ for } \alpha + \beta \notin \Phi \cup \{0\}; \tag{Rel1}\\
\up{x_\alpha(p)}{x_\beta(a)} &= x_\beta(a)\, x_{\alpha + \beta}(N_{\alpha \beta}\, ap) \text{ for } \alpha + \beta \in \Phi. \tag{Rel2}
\end{align*}

If \(d \colon \mathfrak a \to K\) is actually a crossed module, then the relative Steinberg group \(\st(\Phi; K, \mathfrak a)\) is the factor-group of \(\st'(\Phi; K, \mathfrak a)\) by
\[\up{x_\alpha(d(a))}{g} = x_\alpha(a)\, g\, x_\alpha(-a) \text{ for any } g. \tag{Rel3}\]

As in the linear case, there is an isomorphism \(\st'(\Phi; K, \mathfrak a) \cong \st(\Phi; \mathfrak a \rtimes K, \mathfrak a)\) and a decomposition \(\st(\Phi; \mathfrak a \rtimes K) \cong \st'(\Phi; K, \mathfrak a) \rtimes \st(\Phi, K)\). If \(d \colon \mathfrak a \to K\) is a crossed module, then \(\st(\Phi; K, \mathfrak a) \to \st(\Phi, K), x_\alpha(a) \mapsto x_\alpha(d(a))\) is a group-theoretical crossed module and the sequence
\[\st(\Phi; K, \mathfrak a) \to \st(\Phi, K) \to \st(\Phi, K / d(\mathfrak a)) \to 1\]
is exact. Again as in the linear case, we use the elements
\begin{align*}
z_\alpha(a, p) &= \up{x_{-\alpha}(p)}{x_\alpha(a)};\\
z_{\alpha[\beta]}(a, b; p) &= z_\alpha(a, p)\, x_\beta(b)\, x_{\beta - \alpha}(N_{-\alpha, \beta}\, bp)\tag{Z2}\\
&= \up{x_{-\alpha}(p)}{\bigl(x_\alpha(a)\, x_\beta(b)\bigr)} \text{ for } \alpha - \beta \in \Phi;\\
z_{\alpha \oplus \beta}(a, b; p, q) &= z_{\alpha[\alpha - \beta]}(a, N_{-\beta, \alpha}\, aq; p)\, z_{\beta[\beta - \alpha]}(b, N_{-\alpha, \beta}\, bp; q)\tag{Z4}\\
&= \up{x_{-\alpha}(p)\, x_{-\beta}(q)}{\bigl(x_\alpha(a)\, x_\beta(b)\bigr)} \text{ for } \alpha - \beta \in \Phi.
\end{align*}
in the group \(\st'(\Phi; K, \mathfrak a)\), where \(\alpha, \beta, \alpha - \beta \in \Phi\).

In the case \(\Phi = \mathsf A_\ell\) the ring \(\mat(\ell + 1, \mathfrak a)\) is naturally a non-unital \(\mat(\ell + 1, K)\)-algebra. The ring \(\mat(\ell + 1, K)\) has a canonical family of full orthogonal idempotents \(e_1, \ldots, e_{\ell + 1}\). There are canonical isomorphisms
\begin{align*}
\st(\mathsf A_\ell, K) \to \st\bigl(\mat(l + 1, K)\bigr), {}&x_{\mathrm e_i - \mathrm e_j}(p) \mapsto x_{ij}(p e_{ij});\\
\st(\mathsf A_\ell, \mathfrak a) \to \st\bigl(\mat(l + 1, \mathfrak a)\bigr), {}&x_{\mathrm e_i - \mathrm e_j}(a) \mapsto x_{ij}(a e_{ij});\\
\st'(\mathsf A_\ell; K, \mathfrak a) \to \st'\bigl(\mat(\ell + 1, K), \mat(\ell + 1, \mathfrak a)\bigr), {}&z_{\mathrm e_i - \mathrm e_j}(a, p) \mapsto z_{ij}(a e_{ij}, p e_{ji})
\end{align*}
If \(d \colon \mathfrak a \to K\) is a crossed module, then \(d \colon \mat(\ell + 1, \mathfrak a) \to \mat(\ell + 1, K)\) is also a crossed module, so there is a canonical isomorphism
\[\st(\mathsf A_\ell; K, \mathfrak a) \cong \st\bigl(\mat(\ell + 1, K), \mat(\ell + 1, \mathfrak a)\bigr).\]
Of course, \(\st\bigl(\mat(\ell + 1, K)\bigr) = \st\bigl(\mat(\ell + 1, K), \mathsf A_\ell\bigr)\) in our notation from section \ref{s-root-elim}.

\begin{lemma}\label{n-rel}
The structure constants satisfy the relations
\[
N_{\alpha \beta}
= -N_{\beta \alpha}
= -N_{-\alpha, -\beta}
= N_{-\beta, -\alpha}
= N_{\beta, -\alpha - \beta}
= N_{-\alpha - \beta, \alpha}
\]
if \(\alpha, \beta, \alpha + \beta \in \Phi\).
\end{lemma}
\begin{proof}
This is precisely \cite[(14.2)--(14.6)]{ChevElem} (recall that \(N_{\alpha \beta} \in \{-1, 1\}\) in our case).
\end{proof}

Let \(\mathcal G\) be the family of all root subsystems of \(\Phi\) of types \(\mathsf A_1\) and \(\mathsf A_3\) ordered by inclusion. For every \(\Psi \in \mathcal G\) there are natural \(\st(\Psi, K)\)-equivariant homomorphisms
\[
\st'(\Psi; K, \mathfrak a) \to \st'(\Phi; K, \mathfrak a), \quad
\st(\Psi; K, \mathfrak a) \to \st(\Phi; K, \mathfrak a)
\]
given by \(x_\alpha(a) \mapsto x_\alpha(a)\).

\begin{lemma}\label{colimit}
If the rank of \(\Phi\) is at least \(3\), then
\[\st'(\Phi; K, \mathfrak a) = \mathrm{colim}_{\Psi \in \mathcal G} \st'(\Psi; K, \mathfrak a), \quad
\st(\Phi; K, \mathfrak a) = \mathrm{colim}_{\Psi \in \mathcal G} \st(\Psi; K, \mathfrak a).\]
Here both colimits are taken in the category of groups.
\end{lemma}
\begin{proof}
By \cite[theorem 9]{CentralityE} the second equality holds if \(\mathfrak a\) is an ideal of \(K\). Hence the first equality holds for arbitrary non-unital commutative \(K\)-algebra \(\mathfrak a\). To prove the second equality in full generality, note that \(\st'(\Psi; K, \mathfrak a)\) is generated by the elements \(z_\alpha(a, p)\) (for example, by \cite[lemma 4]{CentralityE}). Hence \(\st(\Psi; K, \mathfrak a)\) is the factor-group of \(\st'(\Psi; K, \mathfrak a)\) by the identities \(\up{x_\alpha(d(a))}{z_\alpha(b, p)} = x_\alpha(a)\, z_\alpha(b, p)\, x_\alpha(-a)\) and \(z_\alpha(b, p + d(a)) = x_{-\alpha}(a)\, z_\alpha(b, p)\, x_{-\alpha}(-a)\). Both of them already hold in \(\st(\{-\alpha, \alpha\}; K, \mathfrak a)\).
\end{proof}

In our second main result we actually do not need all the identities from the next lemma. But they are useful to prove that various homomorphisms from the relative Steinberg group are well-defined as in the proof of lemma \ref{elim-bij}.

\begin{lemma}\label{z-rel-st}
The elements \(z_\alpha(a, p)\) satisfy the relations
\begin{align*}
z_\alpha(a + a', p) &= z_\alpha(a, p)\, z_\alpha(a', p); \tag{Add1}\\
z_{\alpha[\beta]}(a + a', b + b'; p) &= z_{\alpha[\beta]}(a, b; p)\, z_{\alpha[\beta]}(a', b'; p); \tag{Add2}\\
z_{\alpha \oplus \beta}(a + a', b + b'; p, q) &= z_{\alpha \oplus \beta}(a, b; p, q)\, z_{\alpha \oplus \beta}(a', b'; p, q); \tag{Add3}\\
\up{z_\alpha(c, r)}{\bigl(x_{\alpha + \beta}(a)\, x_\beta(b)\bigr)} &= x_{\alpha + \beta}(a + \eps bc - acr)\, x_\beta(b + bcr - \eps acr^2) \text{ for } \alpha + \beta \in \Phi \tag{Conj1}\\
\bigl[x_{-\beta}(\eps ap)\, x_\alpha(a), x_\beta(b)\, x_{-\alpha}(-\eps bp)\bigr] &= z_{\alpha + \beta}(ab, p) \text{ for } \alpha + \beta \in \Phi; \tag{Mult}\\
\bigl[z_\alpha(a, p), z_\beta(b, q)\bigr] &= 1 \text{ for } \alpha \perp \beta; \tag{Dis}\\
z_{\alpha \oplus \beta}(a, b; p, q) &= z_{\beta \oplus \alpha}(b, a; q, p); \tag{Sym}\\
\up{z_\alpha(c, r)}{z_{\alpha + \beta \oplus \beta}(a, b; p, q)} &=\\
\up{x_{-\alpha - \beta}(cpr + \eps cqr^2)\, x_{-\beta}(-\eps cp - cqr)}{z}_{\alpha + \beta \oplus \beta}&(a + \eps bc - acr, b + bcr - \eps acr^2; p, q)\\
&\text{ for } \alpha + \beta \in \Phi \tag{Conj2}\\
z_{\alpha + \beta \oplus \beta}(a, \eps aq; r - \eps pq, p) &= z_{\alpha \oplus \alpha + \beta}(-\eps ap, a; q, r) \text{ for } \alpha + \beta \in \Phi. \tag{HW}
\end{align*}
in \(\st'(R, A)\), where \(\eps = N_{\alpha \beta}\) is a structure constant. If \(d \colon \mathfrak a \to K\) is a crossed module, then they also satisfy
\begin{align*}
z_\alpha\bigl(a, p + d(b)\bigr) &= x_{-\alpha}(b)\, z_\alpha(a, p)\, x_{-\alpha}(-b); \tag{Rel4}\\
\up{z_\alpha(c, r)}{z_{\alpha + \beta \oplus \beta}(a, b; p, q)} &= z_{\alpha + \beta \oplus \beta}(a + \eps bc - acr, b + bcr - \eps acr^2; p + d(cpr + \eps cqr^2), q - d(\eps cp + cqr))\\
&\text{ for } \alpha + \beta \in \Phi. \tag{Conj2'}\\
\end{align*}
\end{lemma}
\begin{proof}
This directly follows from lemma \ref{z-rel}.
\end{proof}

\begin{theorem}\label{rel-ade}
Let \(K\) be a commutative unital ring, \(\mathfrak a\) be a commutative non-unital \(K\)-algebra, \(\Phi\) be one of the root systems \(\mathsf A_\ell\), \(\mathsf D_\ell\), \(\mathsf E_\ell\) with \(\ell \geq 3\). As in the linear case, we use the notation \(x_\alpha(a) = z_\alpha(a, 0)\), \fref{z2}, and \fref{z4}. Then \(\st'(\Phi; K, \mathfrak a)\) as an abstract group is given by the generators \(z_\alpha(a, p)\) and the relations \fref{add1}, \fref{dis}, \fref{conj2}, \fref{hw}. If \(d \colon \mathfrak a \to K\) is a crossed module, then \(\st(\Phi; K, \mathfrak a)\) is the factor-group of \(\st'(\Phi; K, \mathfrak a)\) by \fref{rel4}. In this presentation of \(\st(\Phi; K, \mathfrak a)\) the axiom \fref{conj2} may be replaced by \fref{conj2'}.
\end{theorem}
\begin{proof}
Let \(\stabs'(\Phi; K, \mathfrak a)\) be the abstract group from the statement with a canonical homomorphism
\[\stabs'(\Phi; K, \mathfrak a) \to \st'(\Phi; K, \mathfrak a)\]
from lemma \ref{z-rel-st}. It is bijective if \(\Phi = \mathsf A_\ell\) by theorem \ref{rel-lin}.

Let us show that \(\stabs'(\Phi; K, \mathfrak a) = \mathrm{colim}_{\Psi \in \mathcal G} \st'(\Psi; K, \mathfrak a)\), where \(\st'(\Psi; K, \mathfrak a) \cong \stabs'(\Psi; K, \mathfrak a)\) canonically maps to \(\stabs'(\Phi; K, \mathfrak a)\) if \(\Psi\) is of type \(\mathsf A_3\). It suffices to show that if \(\Psi, \Psi' \subseteq \Phi\) are of type \(\mathsf A_3\),  \(\alpha \in \Psi \cap \Psi'\), and \(g \in \st(\{-\alpha, \alpha\}; K, \mathfrak a)\), then the two images of \(g\) in \(\stabs'(\Phi; K, \mathfrak a)\) coincide. By \cite[lemma 2]{CentralityE} we may assume that \(\Psi \cap \Psi'\) is of type \(\mathsf A_2\). Then the image of \(g\) in \(\st'(\Psi \cap \Psi'; K, \mathfrak a)\) lies in the image of \(\stabs'(\Psi \cap \Psi'; K, \mathfrak a) \to \st'(\Psi \cap \Psi'; K, \mathfrak a)\) by theorem \ref{rel-lin}. This implies the claim.

Hence \(\stabs'(\Phi; K, \mathfrak a) \to \st'(\Phi; K, \mathfrak a)\) is an isomorphism. The proof for any of the two presentations of \(\st(\Phi; K, \mathfrak a)\) is the same.
\end{proof}

The theorem gives a slight strengthening of lemma \ref{z-rel-st}: all relations between the generators \(z_\alpha(a, p)\) actually comes from root subsystems of ranks \(1\) and \(2\) (i.e. \(\mathsf A_1\), \(\mathsf A_1 \times \mathsf A_1\), and \(\mathsf A_2\)).

\section{Factoring out transvections}

In this section we study the factor-groups of \(\st(R, A)\) and \(\st(\Phi; K, \mathfrak a)\) by the images of \(\st(A)\) and \(\st(\Phi; \mathfrak a)\).

In the linear case the group \(\st(A)\) usually cannot be a crossed module over \(\st(R)\) satisfying natural properties: its image in \(\glin(R)\) is a subgroup of
\[\{g \in \glin(R) \mid e_i g e_j \in A \text{ for } i \neq j, e_i g e_i \in A^2\},\]
so this image does not contain \(\up{x_{ji}(p)}{x_{ij}(a)}\) for generic \(p\) and \(a\). Similarly, the unrelativized Steinberg group is not invariant under root elimination, i.e. \(F_\Psi \colon \st(A, \Phi / \Psi) \to \st(A, \Phi)\) is not surjective in general even if the rank of \(\Phi / \Psi\) is large.

However, it turns out that at least the image of \(\st(A)\) in \(\st(R, A)\) is normal. In the following theorem we use \(\stabs(R, A)\) instead of \(\st(R, A)\) to cover the case \(n = 3\).

\begin{theorem}\label{f-trans-lin}
Let \(R\) be a ring with a complete family of \(n \geq 3\) full orthogonal idempotents, \(d \colon A \to R\) be a crossed module. Then the image of \(\st(A) \to \stabs(R, A), x_{ij}(a) \mapsto x_{ij}(a)\) is normal with abelian factor-group. This factor-group is the abelian group with generators \(\overline z_{ij}(a, p)\) (the images of \(z_{ij}(a, p)\)) and relations
\begin{align*}
\overline z_{ij}(a + b, p) &= \overline z_{ij}(a, p) + \overline z_{ij}(b, p); \tag{FT1} \label{ft1}\\
\overline z_{ij}(a, p + q) &= \overline z_{ij}(a, p) + \overline z_{ij}(a, q); \tag{FT2} \label{ft2}\\
\overline z_{ij}(ab, p) &= 0 \text{ for } a \in e_i A \text{ and } b \in A e_j; \tag{FT3} \label{ft3}\\
\overline z_{ij}(a, d(b)) &= 0; \tag{FT4} \label{ft4}\\
\overline z_{ij}(a, pq) &= \overline z_{ik}(ap, q) + \overline z_{kj}(qa, p) \text{ if } i, j, k \text{ are distinct}. \tag{FT5} \label{ft5}
\end{align*}
\end{theorem}
\begin{proof}
Let \(N \leq \stabs(R, A)\) be the image of \(\st(A)\), i.e. the subgroup generated by all \(x_{ij}(a)\). First of all, we have to show that \(N\) is normal. Up to transposition this follows from \fref{dis}, \fref{conj1}, \fref{mult}, and the following identities (they are special cases of \fref{conj2}):
\begin{align*}
\up{z_{ij}(c, r)}{\bigl(x_{ij}(-aq)\, x_{ik}(a)\bigr)} &= \up{x_{ki}(qrcr)\, x_{kj}(-qrc)}{z_{[i]j, k}\bigl(a - cra, -rcra; q\bigr)};\\
\up{z_{ij}(c, r)}{\bigl(x_{ji}(-bp)\, x_{jk}(b)\bigr)} &= \up{x_{ki}(pcr)\, x_{kj}(-pc)}{z_{[j]i, k}\bigl(b + rcb, cb; p\bigr)}.
\end{align*}

We use the additive notation for \(\stabs(R, A) / N\). The identity \fref{ft1} is the same as \fref{add1}. By \fref{mult}, \(z_{ij}(ab, p) \in N\) for \(a \in e_i A e_k\) and \(b \in e_k A e_j\) if \(i, j, k\) are distinct. Since \(e_k\) is a full idempotent, \fref{ft3} follows. The identity \fref{ft4} follows from \fref{rel4}.

By \fref{ft1} and \fref{ft3}, the images of \fref{conj2} and \fref{hw} in \(\stabs(R, A) / N\) are
\begin{align*}
\overline z_{ij}(c, r) + \overline z_{ik}(a, p) + \overline z_{jk}(b, q) - \overline z_{ij}(-c, r) &= \overline z_{ik}(a, p) + \overline z_{jk}(b, q);\\
\overline z_{ik}(a, r - pq) + \overline z_{jk}(qa, p) &= \overline z_{ij}(-ap, q) + \overline z_{ik}(a, r).
\end{align*}
Hence any two generators \(\overline z_{ij}(a, p)\) and \(\overline z_{kl}(b, q)\) commute unless \(\{i, j\} = \{k, l\}\). Taking \(r = pq\) in the image of \fref{hw}, we get \fref{ft5}. Hence the image of \fref{hw} may be written as
\[\overline z_{ik}(a, r - pq) + \overline z_{ik}(a, pq) = \overline z_{ik}(a, r),\]
and \fref{ft2} follows. The commutativity of \(\stabs(R, A) / N\) follows from \fref{ft2}, \fref{ft5}, and the already proved cases of commutativity.

Conversely, it is easy to see that the images of \fref{add1}, \fref{dis}, \fref{conj2}, \fref{hw}, \fref{rel4} follow from \fref{ft1}--\fref{ft5}.
\end{proof}

\begin{theorem}\label{f-trans-chev}
Let \(K\) be a commutative unital ring, \(d \colon \mathfrak a \to K\) be a commutative crossed module, \(\Phi\) be any of the root systems \(\mathsf A_l\), \(\mathsf D_l\), \(\mathsf E_l\) with \(l \geq 3\). Then the image of \(\st(\Phi; \mathfrak a) \to \st(\Phi; K, \mathfrak a), x_\alpha(a) \mapsto x_\alpha(a)\) is normal with abelian factor-group. This factor-group is the abelian group with generators \(\overline z_\alpha(a, p)\) (the images of \(z_\alpha(a, p)\)) and relations
\begin{align*}
\overline z_\alpha(a + b, p) &= \overline z_\alpha(a, p) + \overline z_\alpha(b, p); \tag{FT1}\\
\overline z_\alpha(a, p + q) &= \overline z_\alpha(a, p) + \overline z_\alpha(a, q); \tag{FT2}\\
\overline z_\alpha(ab, p) &= 0 \text{ for } a, b \in \mathfrak a; \tag{FT3}\\
\overline z_\alpha(a, d(b)) &= 0; \tag{FT4}\\
\overline z_{\alpha + \beta}(a, pq) &= z_\alpha(ap, q) + z_\beta(qa, p) \text{ if } \alpha + \beta \in \Phi. \tag{FT5}
\end{align*}
\end{theorem}
\begin{proof}
This follows from lemma \ref{colimit} and theorem \ref{f-trans-lin}. The structure constants in all summands in \fref{ft5} are the same, so they may be omitted by \fref{ft1} and \fref{ft2}.
\end{proof}

In conclusion let us list several further problems.

\begin{itemize}
\item Give an explicit presentation of relative unitary Steinberg groups and relative Steinberg groups of type \(\mathsf F_4\). 
\item Give an explicit presentation of birelative Steinberg groups such as \(\st(\Phi; K, \mathfrak a) \otimes \st(\Phi; K, \mathfrak b)\). Here \(\otimes\) means the non-abelian tensor product of crossed modules over \(\st(\Phi, K)\). If \(\Phi\) is simple laced of rank at least \(2\), then by \cite{BirelChev} the image of this tensor product in the simply connected Chevalley group \(G(\Phi, R)\) (i.e. the commutator \([\elin(\Phi; K, \mathfrak a), \elin(\Phi; K, \mathfrak b)]\)) is generated by the image of \(\st(\Phi; K, \mathfrak a \otimes \mathfrak b)\) and the elementary commutators \(y_\alpha(a, b) = [x_\alpha(a), x_{-\alpha}(b)]\) for \(a \in \mathfrak a\) and \(b \in \mathfrak b\) (they are the images of \(x_\alpha(a) \otimes x_{-\alpha}(b)\)). Moreover, there are nice relations between \(y_\alpha(a, b)\) modulo \(\elin(\Phi; K, \mathfrak a \mathfrak b)\). See also \cite{Multirel, BirelLin, BirelUnit} for another such results for elementary groups.
\item Calculate explicitly the relative linear Steinberg group \(\st(R, A)\) and the kernel \(\klin_2(R, A) = \mathrm{Ker}(\st(R, A) \to \glin(A))\) if \(A\) is contained in the Jacobson radical of \(A \rtimes R\), and similarly for simply laced Steinberg groups. Such a description exists for the relative \(\klin_2\)-functor over a commutative ring (\cite{Maazen}) and modulo some subgroup for the relative matrix linear Steinberg group over any ring (\cite{Hochschild}) in the case \(A^2 = 0\).
\item Find some natural factor-group of \(\st(A)\) with a structure of a crossed module over \(\st(R, A)\), and similarly for simply laced Steinberg groups. By theorems \ref{f-trans-lin} and \ref{f-trans-chev} the image of the unrelativized Steinberg group is normal in the relative Steinberg group, so at least some such factor-group exists (i.e. the image itself), although we do not know its explicit presentation. The unrelativized Steinberg group cannot be itself a crossed module in general: consider the field \(\mathbb F_2\) with a crossed module \(0 \colon V \to \mathbb F_2\) over it for some finite dimensional vector space \(V_{\mathbb F_2}\). The group \(\st(\mathsf A_\ell; \mathbb F_2, V)\) is finite for \(\ell \geq 2\) (for example, by \cite{Maazen}), hence every finitely generated crossed module over it is of polynomial growth. But the unrelativized group \(\st(\mathsf A_\ell, V)\) (with zero product on \(V\)) is the \(\frac{\ell (\ell + 1)}2\)-th direct power of the free product \(V * V\), so it has exponential growth if \(\dim(V) > 1\).
\end{itemize}

\bibliographystyle{plain}  
\bibliography{references}

\end{document}